\documentclass[12pt,onecoulme]{article}
\usepackage{amsfonts}
\usepackage{mathrsfs}
\usepackage{amssymb}
\usepackage{color, amsmath,amssymb, amsfonts, amstext,amsthm, latexsym}

\usepackage{amssymb, epsfig, amssymb, latexsym}
\usepackage{amsmath}
\usepackage{graphicx}
\usepackage{longtable}

\numberwithin{equation}{section}

\allowdisplaybreaks

\newtheorem{theorem}{Theorem}[section]
\newtheorem{lemma}{Lemma}[section]

\title{Invariant foliations for stochastic partial differential equations with dynamic boundary conditions}

\author{Zhongkai Guo$^1$
 \\
  1. School of Mathematics and Statistics\\ Huazhong University of Science and Technology\\
   Wuhan 430074, China \\ \emph{nj4102008@gmail.com}}

\begin{document}

\maketitle

\pagestyle{plain}

\begin{abstract}
Invariant foliations are complicated random sets useful for
describing and understanding the qualitative behaviors of nonlinear
dynamical systems. We will consider invariant foliations for
stochastic partial differential equation with dynamical boundary
condition.
\medskip


\end{abstract}


\section{Introduction}

\quad\, Invariant foliations are often used to study the qualitative
properties of a flow or semiflow nearby invariant sets. They are
extremely useful because they can be used to track the asymptotic
behavior of solutions and to provide coordinates in which systems of
differential equations may be decoupled and normal forms
derived.\par Invariant foliations as well as invariant manifolds,
provide geometric structures for understanding the qualitative
behaviors of nonlinear dynamical systems. There are some works on
invariant foliations for finite dimensional random dynamical systems
by Arnold\cite{Arnold}, Wanner\cite{Wanner}. For infinite
dimensional, Lu and Schmalfuss\cite{LS} considered the invariant
foliation for stochastic partial differential equations, Chen, Duan
and Zhang\cite{CDZ} concern the slow foliation for slow-fast
stochastic evolutionary system, Sun, Kan and Duan\cite{SKD} examined
the approximation of invariant foliation for stochastic dynamical
systems. There are also some other papers about the existence of
invariant manifolds  we recommend  \cite{DLS}, \cite{LS}, and so on.
\par
The intention of this article is to study the dynamics of the
following class of randomly perturbed parabolic partial differential
equations with dynamical boundary conditions.

\begin{equation}\label{e1.1}
\left\{\begin{array}{ll}
\frac{\partial u}{\partial t}-\sum_{k, j=1}^{n}\partial_{x_{k}}(a_{kj}(x)\partial_{x_{j}}u)+a_{0}(x)u=f(u)+\epsilon \dot{W_{1}} &\mbox{on D}\,,\\
\frac{\partial u}{\partial t}+\sum_{k, j=1}^{n}\nu_{k}a_{kj}(x)\partial_{x_{j}}u+c(x)u=g(u)+\epsilon \dot{W_{2}} &\mbox{on $\partial D$}\,,\\
u(0, x)=u_{0}(x) &\mbox{on $D\times \partial D$\,,}
\end{array}\right.
\end{equation}
where $W_{1}, W_{2}$ are two independent Wiener process,
 $0<\epsilon\ll1$ and \(\nu=(\nu_{1}, \cdots, \nu_{n})\) is the outer normal to $\partial D$\,. Moreover, $f$ and $g$ are nonlinear
 terms, and $a_{k,j},\, a_{0},\, c$ are given coefficients.
 \par  A simple example of system (\ref{e1.1}) is the following problem

\begin{equation}\label{e1.2}
\left\{\begin{array}{ll}
\frac{\partial u}{\partial t}-\Delta u=f(u)+\epsilon\dot{W_{1}} &\mbox{on D}\,,\\
\frac{\partial u}{\partial t}+\frac{\partial u}{\partial \nu}=c(x)u+g(u)+\epsilon\dot{W_{2}} &\mbox{on $\partial D$}\,,\\
u(0, x)=u_{0}(x) &\mbox{on $D\times \partial D$\,.}
\end{array}\right.
\end{equation}
Boundary conditions of this type are usually called stochastic
dynamical boundary condition, because on the boundary it involves
the It\^o differential of the unknown function $u$ with respect to
time. Deterministic parabolic systems with dynamical boundary
conditions arise in hydrodynamics and the heat transfer theory and
were studied by many authors, see \cite{Escher} and
\cite{Hintermann}.

Boundary value problems with ``classical" boundary conditions are
frequently converted into an abstract cauchy problem, see \cite{EN}
or \cite{Goldstein}. The abstract problem can be seen as an
evolution equation in suitable Banach space, driven by a given
operator, and boundary values are necessary for defining the domain
of this operator. Here we concerned with nonclassical boundary
conditions, similar problems are already solved in the literature
with different cases, see \cite{AB} and \cite{CS04}. A different
approach to evolution problem with dynamical boundary conditions
were recently proposed in functional analysis, motivated by the
study on matrix operator theory \cite{Engel}, it is useful to
translate non-classical boundary value problems in an abstract
setting then the semigroup techniques are available. In this paper
we also by this approach.

\par
Our  main goal in this paper is to show  the existence of invariant
foliation for problem (\ref{e1.1}) and study an approximation of it.
\par
This paper is organized as follows. In section $2$\,, we will review
some basic concepts of dynamical boundary problems, random dynamical
systems, invariant foliations. The result on the existence of
invariant foliations for stochastic differential equation with
dynamical boundary condition is described  in section $3$\,. In the
section $4$\,, we present an asymptotic approximation for random
invariant foliations of equation (\ref{e1.1}).

\section{Preliminaries}
\subsection{Dynamical boundary value problems}

\quad\, In this section we are going to collect some results about
the theory of parabolic partial differential equation with dynamical
boundary condition.  For more details we refer to Amann and
Escher\cite{AE}, Chueshov and Schmalfuss\cite{CS}. Let \(D\subset
R^{n}\) be a bounded $C^{\infty}$-smooth domain with the boundary
$\partial D$\,, we use $W_{p}^{s}(D)$ and $W_{p}^{s}(\partial D)\,,
s>0$\,, for the notations of Sobolev-slobodetski spaces, see
\cite{BS}. In particular, we denote $H^{s}(D):=W_{2}^{s}(D)$\,,
$H^{s}(\partial D):=W_{2}^{s}(\partial D)$  and \[\mathbb{L}_{{p}\,,
{q}}(D):=L_{p}(D)\times L_{q}(\partial D)\,,\quad
\mathbb{L}_{p}(D):=\mathbb{L}_{p, p}(D)\,, \quad p\,, q\geq 1\,,\]
\[\mathcal{V}:=\{(u, \gamma u)\in H^{1}(D)\times H^{\frac{1}{2}}(\partial D):u=\gamma u \quad on \quad \partial D\}\,,\]
where the $\gamma$ denote the trace operator. For the norm of
$\mathbb{L}_{{p}, {q}}(D)$\,, we set
\[ \|u\|_{\mathbb{L}_{{p}, {q}}}:=\|u\|_{L_{p}(D)}+\|\gamma u\|_{L_{q}(\partial D)}\,, \quad u=(u, \gamma u)\in\mathbb{L}_{{p}, {q}}(D)\,, \]
and similarly for the other spaces. we set
$\mathcal{H}:=\mathbb{L}_{2}(D)$\,. Denote by $\|\cdot\|$ and
$(\cdot\,, \cdot)$ the norm and the inner product in
$\mathcal{H}$\,, then we have $\mathcal{V}$ is densely and compactly
embedded in $\mathcal{H}$\,.
\par
Now we consider the differential operators
\[\mathcal{A}(x, \partial):=-\sum_{k, j=1}^{n}\partial_{x_{k}}(a_{kj}(x)\partial_{x_{j}})+a_{0}(x)\]
and
\[\mathcal{B}(x, \partial):=\sum_{k, j=1}^{n}\nu_{k}a_{kj}(x)\partial_{x_{j}}+c(x)\,,\]
where $\nu=(\nu_{1},\cdots,\nu_{n})$ is the outer normal to
$\partial D$\,. We assume that $a_{kj}(x)$ and $a_{0}(x)$ are
$C^{\infty}(\bar{D})$ functions and $a_{0}(x)>0$ for almost all
$x\in \bar{D}$\,. Let the matrix ${a_{kj}(x)}|_{k, j=1}^{d}$ be
symmetric and uniformly positive definite. The function $c(x)$ is
positive and belongs to $C(\partial D)$\,.
\par
We consider  the continuous symmetric positive bilinear form on the
space $\mathcal{V}$\,.
\[a(U, V)=\sum_{k, j=1}^{d}\int_{D}a_{kj}(x)\partial_{x_{k}}u(x)\partial_{x_{j}}v(x)dx+
\int_{D}a_{0}(x)u(x)v(x)dx+\int_{\partial D}c(s)\,\gamma u(s)\,
\gamma v(s)ds\,,\]
 where $U=(u\,, \gamma u)\,,\, V=(v\,, \gamma v)\in \mathcal{V}$\,. Following the Lax-Milgram theory
 this bilinear form generates a positive self adjoint operator $A$ in $\mathcal{H}$\,.
 By the Green formula this operator $A$ is related to the pair $(\mathcal{A}(x, \partial)\,, \mathcal{B}(x, \partial))$\,.
Then there exists an orthonormal basis $\{E_{k}\}_{k\in \mathbb{N}}$
in $\mathcal{H}$ such that
\[ AE_{k}=\lambda_{k}E_{k}\,, \quad k=1, 2, \cdots, \quad 0<\lambda_{1}\leq\lambda_{2}\cdots, \quad \lim_{k\rightarrow\infty}\lambda_{k}=\infty\,,\]
where $E_{k}$ is the eigenfunctions of $A$  has the form
$E_{k}=(e_{k};\gamma[e_{k}])\,,\, k=1, 2, 3, \cdots$\,, where
$e_{k}\in C^{\infty}(\bar{D})$\,. and eigenvalues of $A$ have a
finite multiplicity.
\par For the domain of definition of $A$\,,\[D(A)\subset H^{\frac{3}{2}}(D)\times H^{1}(\partial D)\,,\]
we conclude that $-A$ is the generator of a $C_{0}$-semigroup
$(S(t))_{t\in R^{+}}$\,. More detail of $A$ see \cite{AE}.\par By
the above analysis,  the abstract form of the systems (\ref{e1.1})
is
\begin{equation}\label{e2.1}
d\hat{X}+A\hat{X}dt=F(\hat{X})dt+\epsilon dW\,, \quad
\hat{X}_{0}=\hat{X}(0)\in \mathcal{H}\,,
\end{equation}
where $\hat{X}=\left( \begin{array}{c} u \\ \gamma u \end{array}
\right)\,.$

\subsection{Random dynamical systems}
\quad\, We will recall some basic concepts in random dynamical
systems in this section (see \cite{DLS}). Let $(\Omega, \mathcal{F},
\mathbb{P})$ be a probability space. A flow \(\theta\) of mappings
\(\{\theta_t\}_{t\in \mathbb{R}}\) is defined on the sample space
$\Omega$ such that
\[\theta: \mathbb{R}\times \Omega\to \Omega\,, \quad
\theta_0=id\,, \quad \theta_{t_1}\theta_{t_2}=\theta_{t_1+t_2}\,. \]
for \(t_1\,, t_2 \in \mathbb{R}\)\,. This flow is supposed to be
$(\mathcal{B}(\mathbb{R})\otimes\mathcal{F}\,,
\mathcal{F})$-measurable, where $\mathcal{B}(\mathbb{R})$ is the
$\sigma$-algebra of Borel sets on the real line $\mathbb{R}$\,. To
have this measurability, it is not allowed to replace $\mathcal{F}$
by its $\mathbb{P}$-completion $\mathcal{F}^{\mathbb{P}}$ (see
Arnold \cite{Arnold} p. 547). In addition, the measure $\mathbb{P}$
is assumed to be ergodic with respect to $\{\theta_t\}_{t\in
\mathbb{R}}$\,. Then $(\Omega\,, \mathcal{F}\,, \mathbb{P}\,,
\mathbb{R}\,, \theta)$ is called a metric dynamical system.

For our applications, we will consider a special but very important
metric dynamical system induced by the Brownian motion. Let $W(t)$
be a two-sided Wiener process with trajectories in the space
$C_0(\mathbb{R}\,, \mathbb{H})$ of real continuous functions defined
on $\mathbb{R}$\,, taking zero value at $t = 0$\,. This set is
equipped with the compact open topology. On this set we consider the
measurable flow $\theta = \{\theta_t\}_{t\in \mathbb{R}}$\,, defined
by \[\theta_t\omega =\omega(\cdot+t)-\omega(t)\,,\,
\omega\in\Omega\,,\, t\in R\,.\] The distribution of this process
generates a measure on $\mathcal{B}(C_0(\mathbb{R}\,, \mathbb{H}))$
which is called the Wiener measure. Note that this measure is
ergodic with respect to the above flow, (see the Appendix in Arnold
\cite{Arnold}). Later on we will consider, instead of the whole
$C_0(\mathbb{R}\,, \mathbb{H})$\,, a $\{\theta_t\}_{t\in
\mathbb{R}}$-invariant subset $\Omega\subset
C_0(\mathbb{R}\,,\mathbb{H})$) of $\mathbb{P}$-measure one and the
trace $\sigma$-algebra $\mathcal{F}$ of
$\mathcal{B}(C_0(\mathbb{R}\,, \mathbb{H}))$ with respect to
$\Omega$\,. A set $\Omega$ is called {$\{\theta_t\}_{t\in
\mathbb{R}}$ invariant if $\theta_t\,\Omega = \Omega$ for $t \in
\mathbb{R}$\,. On $\mathcal{F}$\,, we consider the restriction of
the Wiener measure also denoted by $\mathbb{P}$\,.

The dynamics of the system on the state space $H$ over the flow
$\theta$ is described by a cocycle. For our applications it is
sufficient to assume that $(H, d)$ is a complete metric space. A
cocycle $\phi$ is a mapping:
$$
\phi : \mathbb{R}^+\times\Omega\times H\to H,
$$
which is $(\mathcal{B}(\mathbb{R})\otimes \mathcal{F} \otimes
\mathcal{B}(H)\,, \mathcal{F})$-measurable such that
$$
\begin{array}{l}
\phi(0\,, \omega\,, x) = x \in H\,,\\
\phi(t_1 + t_2\,, \omega\,, x) = \phi(t_2\,, \theta_{t_1}\omega\,,
\phi(t_1\,, \omega\,, x))\,,
\end{array}
$$
for $t_1, t_2 \in \mathbb{R}^+,\, \omega \in \Omega$ and $x\in H$\,.
Then $\phi$\,, together with the metric dynamical system $\theta$\,,
forms a random dynamical system.

\section{Existence of invariant foliation}
\quad\, Invariant foliation is about quantifying certain sets called
fibers or leaves in state space. A fiber consists of all points
starting from which the dynamical orbits have similar asymptotic
behavior. These fibers are thus building blocks for understanding
dynamics. The definitions of  stable fiber and unstable fiber are as
follows (see\cite{CDZ}).

Let $H$ be a state space, and
$\psi(\cdot\,,\cdot\,,\cdot):\,R\times\Omega\times H\rightarrow H$
be a random dynamical system.

\begin{enumerate}
  \item  We say that $\mathcal{W}_{\beta s}(x, \omega)$ is a $\beta-stable fiber $ passing through $x\in H$ with $\beta \in \mathbb{R}^{-}$,\,if
  $\|\psi(t,\omega,x)-\psi(t,\omega,\widetilde{x})\|=\mathrm{O}(e^{\beta t})\,,\forall\, \omega\in \Omega$\,, as $t\rightarrow +\infty$
  \,for any $x,\,\widetilde{x}\in \mathcal{W}_{\beta s}$\,.
  \item   We say that $\mathcal{W}_{\beta u}(x,\omega)$ is a $\beta-unstable fiber $ passing through $x\in H$ with $\beta \in \mathbb{R}^{+}$,\,if
  $\|\psi(t,\omega,x)-\psi(t,\omega,\widetilde{x})\|=\mathrm{O}(e^{\beta t})\,,\forall\, \omega\in \Omega$\,, as $t\rightarrow -\infty$
  for any $x,\,\widetilde{x}\in \mathcal{W}_{\beta u}$\,.
\end{enumerate}

We say a foliation is invariant if it satisfy
$$\phi(t,\omega,\mathcal{W}_{\beta}(x,\omega))\subset
\mathcal{W}_{\beta}(\phi(t,\omega,x)\,,\theta_{t}\omega)\,.$$

Now we consider the existence of invariant foliation for the
following equations
$$
\left\{\begin{array}{ll}
\frac{\partial u}{\partial t}-\sum_{k, j=1}^{n}\partial_{x_{k}}(a_{kj}(x)\partial_{x_{j}}u)+a_{0}(x)u=f(u)+\epsilon\dot{W_{1}} &\mbox{on D},\\
\frac{\partial u}{\partial t}+\sum_{k, j=1}^{n}\nu_{k}a_{kj}(x)\partial_{x_{j}}u+c(x)u=g(u)+\epsilon\dot{W_{2}} &\mbox{on $\partial D$},\\
u(0, x)=u_{0}(x) &\mbox{on $D\times \partial D$},
\end{array}\right.
$$
where $f, g$ are Lipschitz continuous function.\par Take the
abstract form of above systems, we obtain
 \[d\hat{X}+A\hat{X}dt=F(\hat{X})dt+\epsilon dW, \quad \hat{X}_{0}=\hat{X}(0)\in \mathcal{H}\,.\]

Where $\hat{X}(t)=(u(t)\,,\, \gamma u(t))^{T}$,\,
$F(\hat{X})=(f(u)\,,\, g(\gamma u))^{T}$,\, $W=(W_{1},\,
W_{2})^{T}$,\, $\hat{X}(0)=(u(0)\,,\, \gamma u(0))^{T}$\,. By the
Lipschitz continuous of  $f, g$ we can easily obtain  $F$ is also
Lipschitz continuous. Denote the Lipschitz constant by $L_{F}$\,.

 \par

Consider the linear stochastic evolution equation
 \[dZ^{\epsilon}(t)+AZ^{\epsilon}(t)=\epsilon dW(t)\,.\] \par
 Let $X(t)=\hat{X}(t)-Z^{\epsilon}(t)$ then we have
\begin{equation}\label{e3.1}
dX+AXdt=F(X+Z^{\epsilon})dt\,.
\end{equation}
In the following, we assume that

\noindent$\textbf{Hypothesis H1}$
 Projection operators of the linear
(unbounded) operator $A$ satisfies Lemma\ref{l3.1}.\\
\textbf{Hypothesis H2} There is constant $0\leq\alpha<1$ and a
Lipschitz continuous mapping
$$F : D(A^{\alpha})\rightarrow \mathcal{H}.$$
In addition, $F$ satisfies the following condition
$\|F(u)-F(v)\|_{\mathcal{H}}\leq L_{F}\|u-v\|_{D(A^{\alpha})}.$

Take $\mathcal{H}_{1}$ be the space spanned by the eigenfunctions
related to the first $N$ eigenvalues of the positive symmetric
operator $A$\,,  we denote $\pi_{1}$ as the orthogonal projection
relate to $\mathcal{H}_{1}$\,. Similarly, we can describe the
infinite dimensional space $\mathcal{H}_{2}$ given by the span of
the eigenfunction of $\lambda_{N+1},\cdots$\,, the related
projection is denoted by $\pi_{2}$\,. Then we have
$\mathcal{H}=\mathcal{H}_{1}\oplus \mathcal{H}_{2}$\,. For simple,
we define follow notation.
\[\pi_{1}A=A_{1},\, \pi_{2}A=A_{2}\,;\,\pi_{1}F=F_{1}, \,\pi_{2}F=F_{2}\,;\,\pi_{1}X=X_{1}, \,\pi_{2}X=X_{2}\,;\]
Then equation (\ref{e3.1})  can be rewritten as
\begin{equation}\label{e3.2}
\left\{\begin{array}{ll}
dX_{1}=-A_{1}X_{1}+F_{1}(X_{1}, X_{2}, \theta_{t}^{\epsilon}\omega) &\mbox{in $\mathcal{H}_{1}$}\,,\\
dX_{2}=-A_{2}X_{2}+F_{2}(X_{1}, X_{2}, \theta_{t}^{\epsilon}\omega)
&\mbox{in $\mathcal{H}_{2}$}\,.
\end{array}\right.
\end{equation}
\par
Define \[e^{-At}X=\sum_{i=1}^{\infty}e^{-\lambda_{i}t}(X,
E_{i})E_{i}\, , \quad E_{i}=(e_{i};\gamma[e_{i}])\] and
\[\|e^{-At}X\|_{L(\mathcal{H}, D(A^{\alpha}))}=\bigg(
\sum_{i=1}^{\infty}e^{-2\lambda_{i}t}(X,
E_{i})^{2}\lambda_{i}^{2\alpha}\bigg)^{\frac{1}{2}},\]
then we have the following estimate for the semigroup $S(t)$\,.\\

\begin{lemma}\label{l3.1}\textbf{(\cite{Chueshov})} \quad Let $\alpha\in[0, 1)$ be a constant. Then
\begin{enumerate}
  \item  for \(t>0\) \[\|e^{-A_{2}t}\|_{L(\mathcal{H}_{2}, D(A_{2}^{\alpha}))}\leq\bigg(\frac{\alpha}{t^{\alpha}}
  +\lambda_{N+1}^{\alpha}\bigg)e^{-\lambda_{N+1}t}\,. \]
  \item for \(t\leq 0\)
      \[\|e^{-A_{1}t}\|_{L(\mathcal{H}_{1}, D(A_{1}^{\alpha}))}\leq \lambda_{N}^{\alpha}e^{-\lambda_{N}t}\,. \]
\end{enumerate}
\end{lemma}

\par
Define a Banach space for a fixed
$\beta=\lambda_{N}+\frac{2L_{F}}{k}\lambda_{N}^{\alpha}\in(\lambda_{N},
\lambda_{N+1})$\,,
\[\mathcal{C}_{\beta, \alpha}^{i, +}=\{X:[0, \infty)\rightarrow D(A_{i}^{\alpha})\,, \|X\|_{\mathcal{C}_{\beta}^{i, +}}
=\sup\limits_{t\geq0}e^{\beta
t}\|X\|_{D(A_{i}^{\alpha})}\}<\infty\,, \] with the norm
$\|X\|_{\mathcal{C}_{\beta, \alpha}^{i,
+}}=\sup\limits_{t\geq0}e^{\beta t}\|X\|_{D(A_{i}^{\alpha})}$\,, for
$i=1, 2$\,. Set $\mathcal{C}_{\beta, \alpha}^{+}=\mathcal{C}_{\beta,
\alpha}^{1, +}\times \mathcal{C}_{\beta, \alpha}^{2, +}$\,, with
norm
\[ \|(X_{1}, X_{2})\|_{\mathcal{C}_{\beta, \alpha}^{+}}=\|X_{1}\|_{\mathcal{C}_{\beta, \alpha}^{1, +}}+
 \|X_{2}\|_{\mathcal{C}_{\beta, \alpha}^{2, +}}, \quad\quad(X_{1}, X_{2})\in \mathcal{C}_{\beta, \alpha}^{+}. \]
Denote $\Phi(t, \omega, (X_{1, 0}, X_{2, 0}))=\big(X_{1}(t, \omega,
(X_{1, 0}, X_{2, 0}))\,, X_{2}(t, \omega, (X_{1, 0}\,, X_{2,
0}))\big)$, the solution of the random system (\ref{e3.2}) with the
initial condition $\Phi(0, \omega, (X_{1,0},X_{2, 0}))=(X_{1, 0},
X_{2, 0})\,.$
\par
Define the difference of two dynamical orbits as
\[
\begin{array}{ll}
\Psi(t)&=
\Phi(t, \omega, (\widetilde{X}_{1, 0}, \widetilde{X}_{2, 0}))-\Phi(t, \omega, (X_{1, 0}, X_{2, 0}))\\
&=(X_{1}(t, \omega, (\widetilde{X}_{1, 0}, \widetilde{X}_{2, 0}))-
X_{1}(t, \omega, ({X_{1, 0}}, {X_{2, 0}}))\,,\\
& X_{2}(t, \omega, (\widetilde{X}_{1, 0}, \widetilde{X}_{2, 0}))-
X_{2}(t, \omega, ({X_{1, 0}}, {X_{2, 0}})))\\
&:=(U_{1}(t)\,, U_{2}(t))\,,
\end{array}
\]
with initial condition
\[\Psi(0)=(U_{1}(0)\,, U_{2}(0))=(\widetilde{X}_{1, 0}-X_{1, 0}, \widetilde{X}_{2, 0}-X_{2, 0})\,.\]
Then we have
$$
\begin{array}{ll}
X_{1}(t, \omega, (\widetilde{X}_{1, 0}, \widetilde{X}_{2, 0}))=U_{1}(t)+X_{1}(t, \omega, (X_{1, 0}, X_{2, 0}))\, ,\\
X_{2}(t, \omega, (\widetilde{X}_{1, 0}, \widetilde{X}_{2,
0}))=U_{2}(t)+X_{2}(t, \omega, (X_{1, 0}, X_{2, 0}))\, .
\end{array}
$$
Thus $(U_{1}, U_{2})$ satisfies the follow system of equations
$$
\left\{\begin{array}{ll}
\frac{dU_{1}}{dt}=-A_{1}U_{1}+\Delta F_{1}(U_{1}, U_{2}, \theta_{t}\omega)\,,\\
\frac{dU_{2}}{dt}=-A_{2}U_{2}+\Delta F_{2}(U_{1}, U_{2},
\theta_{t}\omega)\,,
\end{array}\right.
$$
where the nonlinearities are
\[
\begin{array}{ll}
\Delta F_{1}(U_{1}, U_{2}, \theta_{t}\omega)=&F_{1}(U_{1}(t)+X_{1}(t, \omega, (X_{1, 0}, X_{2, 0}))\,, U_{2}+X_{2}(t, \omega, (X_{1, 0}, X_{2, 0}))\,, \theta_{t}\omega)\\
&-F_{1}(X_{1}(t, \omega, (X_{1, 0}, X_{2, 0}))\,, X_{2}(t, \omega,
(X_{1, 0}, X_{2, 0}))\,, \theta_{t}\omega)
\end{array}
\]
and
\[
\begin{array}{ll}
\Delta F_{2}(U_{1}, U_{2}, \theta_{t}\omega)=&F_{2}(U_{1}(t)+X_{1}(t, \omega,(X_{1, 0}, X_{2, 0}))\,, U_{2}+X_{2}(t, \omega, (X_{1, 0}, X_{2, 0}))\,, \theta_{t}\omega)\\
&-F_{2}(X_{1}(t, \omega, (X_{1, 0}, X_{2, 0}))\,, X_{2}(t, \omega,
(X_{1, 0}, X_{2, 0}))\,, \theta_{t}\omega)\, .
\end{array}
\]
Define
\[
\begin{array}{ll}
\mathcal{W}_{\beta s}((X_{1, 0}, X_{2, 0},
\omega))=&\{(\widetilde{X}_{1, 0}, \widetilde{X}_{2, 0})\in
D(A_{1}^{\alpha})\times D(A_{2}^{\alpha})| \Phi(t, \omega,
(\widetilde{X}_{1, 0}, \widetilde{X}_{2, 0}))-\\
&\Phi(t, \omega, (X_{1, 0}, X_{2, 0}))\in \mathcal{C}_{\beta,
\alpha}^{+}\}\,.\end{array}
\]
 Then we
will prove $\mathcal{W}_{\beta s}((X_{1, 0}, X_{2, 0})\,, \omega)$
is a fiber of the foliation for the system (\ref{e3.2})\,.

\begin{lemma}\label{l3.2}
{\it Take
$\beta=\lambda_{N}+\frac{2L_{F}}{k}\lambda_{N}^{\alpha}\in(\lambda_{N},
\lambda_{N+1})$ as a positive real number. Then $(\widetilde{X}_{1,
0}, \widetilde{X}_{2, 0})\in\mathcal{W}_{\beta}((X_{1, 0}, X_{2, 0},
\omega))$ if and only if there exists a function
$$\Psi(t)=(U_{1}(t)\,, U_{2}(t))=(U_{1}(t, \omega, (X_{1, 0}, X_{2, 0});U_{2}(0))\,, U_{2}(t, \omega, (X_{1, 0}, X_{2, 0});U_{2}(0)))
\in \mathcal{C}_{\beta, \alpha}^{+}\,,$$ such that
\begin{equation}\label{e3.3}
\Psi(t)=\left( \begin{array}{c} U_{1}(t) \\ U_{2}(t) \end{array}
\right)=\left( \begin{array}{c}
\int_{\infty}^{t}e^{-A_{1}(t-s)}\Delta F_{1}(U_{1}(s)\,, U_{2}(s)\,,
\theta_{s}\omega)ds
 \\ e^{-A_{2}t}U_{2}(0)+\int_{0}^{t}e^{-A_{2}(t-s)}\Delta F_{2}(U_{1}(s)\,, U_{2}(s)\,, \theta_{s}\omega)ds \end{array} \right)\,,
\end{equation}
 where $\Delta F_{1}$ and $\Delta F_{2}$ are defined as above.\par
\begin{proof}
\quad Let $(\widetilde{X}_{1, 0}, \widetilde{X}_{2, 0})\in
\mathcal{W}_{\beta s}((X_{1, 0}, X_{2, 0})\,, \omega)$\,. Using the
variation of constants formula, we have

 \[
\begin{array}{ll}
U_{1}(t)=e^{-A_{1}(t-\tau)}U_{1}(\tau)+\int_{\tau}^{t}e^{-A_{1}(t-s)}\Delta F_{1}(U_{1}(s)\,, U_{2}(s)\,, \theta_{s}\omega)ds\,,\\
U_{2}(t)=e^{-A_{2}t}U_{2}(0)+\int_{0}^{t}e^{-A_{2}(t-s)}\Delta
F_{2}(U_{1}(s)\,, U_{2}(s)\,, \theta_{s}\omega)ds\,.

\end{array}
\]
Note that $\Phi(\cdot)\in \mathcal{C}_{\beta, \alpha}^{+}$\,. For
$\tau >0$ and $\tau>t$
$$
\begin{array}{ll}
\|e^{-A_{1}(t-\tau)}U_{1}(\tau)\|_{\mathcal{C}_{\beta, \alpha}^{1,
+}}&\leq
\sup\limits_{t\geq0}e^{\beta t}\|e^{-A_{1}(t-\tau)}U_{1}(\tau)\|_{D(A_{1}^{\alpha})}\\
&=\sup\limits_{t\geq0}e^{\beta t}\|A_{1}^{\alpha}e^{-A_{1}(t-\tau)}U_{1}(\tau)\|\\
&\leq \sup\limits_{t\geq0}e^{\beta t}\lambda_{N}^{\alpha}e^{-\lambda_{N}(t-\tau)}\|A_{1}^{\alpha}U_{1}(\tau)\| \\
&\leq \lambda_{N}^{\alpha}e^{(\beta-\lambda_{N})t}e^{(\lambda_{N}-\beta)\tau}\|U_{1}(\tau)\|_{\mathcal{C}_{\beta, \alpha}^{1, +}}\\
&\rightarrow 0 \quad as \quad \tau\rightarrow +\infty\,.
\end{array}
$$
Then we conclude that
 \[
\begin{array}{ll}
U_{1}(t)=\int_{\infty}^{t}e^{-A_{1}(t-s)}\Delta F_{1}(U_{1}(s)\,, U_{2}(s)\,, \theta_{s}\omega)ds\,,\\
U_{2}(t)=e^{-A_{2}t}U_{2}(0)+\int_{0}^{t}e^{-A_{2}(t-s)}\Delta
F_{2}(U_{1}(s)\,, U_{2}(s)\,, \theta_{s}\omega)ds\,,
\end{array}
\]
which imply that equation (\ref{e3.3}) holds.\par
 For the converse,
 \[\Psi(t)=(U_{1}(t)\,, U_{2}(t))=(U_{1}(t, \omega, (X_{1, 0}, X_{2, 0});U_{2}(0))\,, U_{2}(t, \omega, (X_{1, 0}, X_{2, 0});U_{2}(0)))
\in \mathcal{C}_{\beta, \alpha}^{+}\,,
\]
where $$\Psi(t)=\Phi(t, \omega, (\widetilde{X}_{1, 0},
\widetilde{X}_{2, 0}))-\Phi(t, \omega, (X_{1, 0}, X_{2, 0}))$$\,.
Then by the definition of $\mathcal{W}_{\beta s}((X_{1, 0},
X_{2,0})\,,\, \omega)$  we have $(\widetilde{X}_{1, 0},
\widetilde{X}_{2, 0})\in\mathcal{W}_{\beta}((X_{1, 0}, X_{2,
0})\,,\, \omega)$, which complete our proof.
\end{proof}}
\end{lemma}

\par
\begin{lemma}\label{l3.3}\quad{\it Take $\beta=\lambda_{N}+\frac{2L_{F}}{k}\lambda_{N}^{\alpha}\in(\lambda_{N},\, \lambda_{N+1})$
as a positive real number, and let $U_{2}(0)=\widetilde{X}_{2,
0}-X_{2, 0}\in D(A_{2}^{\alpha})$\,. Then the system (\ref{e3.3})
has a unique solution $$\Psi(\cdot)=\Psi(\cdot,\, \omega,\, (X_{1,
0},\, X_{2, 0})\,;\,U_{2}(0))\in \mathcal{C}_{\beta, \alpha}^{+}\,.
$$
\par
\begin{proof}
\quad Introduce two operators $\mathfrak{F}_{1}$ and
$\mathfrak{F}_{2}$ satisfying
\[
\begin{array}{lll}
\mathfrak{F}_{1}(\Psi)[t]=\int_{\infty}^{t}e^{-A_{1}(t-s)}\Delta
F_{1}(U_{1}(s)\,, U_{2}(s),
\theta_{s}\omega)ds\,,\\
\mathfrak{F}_{2}(\Psi)[t]=e^{-A_{2}t}U_{2}(0)+\int_{0}^{t}e^{-A_{2}(t-s)}\Delta
F_{2}(U_{1}(s)\,, U_{2}(s)\,, \theta_{s}\omega)ds\,.
\end{array}
\]\par
It is easily to verify that $\mathfrak{F}_{i}$  maps
$\mathcal{C}_{\beta,\, \alpha}^{i,+}$ into itself respectively,
$i=1,2$\,. Pose the operator \\$\mathfrak{F}:\mathcal{C}_{\beta,
\alpha}^{+}\rightarrow \mathcal{C}_{\beta, \alpha}^{+}$ ,
 $\mathfrak{F}(\Psi):=(\mathfrak{F_{1}}(\Psi)\,,\, \mathfrak{F_{2}}(\Psi))$\,. Then $\mathfrak{F}$ is well-defined
 in $\mathcal{C}_{\beta, \alpha}^{+}$\,
 .\\

 \par
 For every $\Psi=(U_{1},\, U_{2})\in \mathcal{C}_{\beta,\, \alpha}^{+}$
 and $\widetilde{\Psi}=(\widetilde{U}_{1}, \widetilde{U}_{2})\in \mathcal{C}_{\beta, \alpha}^{+}$\,, we have
$$
\begin{array}{ll}
&\|\mathfrak{F_{1}}(\Psi)-\mathfrak{F_{1}}(\widetilde{\Psi})\|_{\mathcal{C}_{\beta,
\alpha}^{1,
+}}\\
=&\|\int_{\infty}^{t}e^{-A_{1}(t-s)}[\Delta F_{1}(U_{1}(s)\,,
U_{2}(s), \theta_{s}\omega)ds -\Delta F_{1}(\widetilde{U}_{1}(s)\,,
\widetilde{U}_{2}(s)\,, \theta_{s}\omega)ds]\|_{\mathcal{C}_{\beta,
\alpha}^{1,
+}}\\
=&\|\int_{\infty}^{t}e^{-A_{1}(t-s)}[F_{1}(U_{1}(s)+X_{1}(s, \omega,
(X_{1, 0}, X_{2, 0}))\,, U_{2}(s)+X_{2}(s, \omega, (X_{1, 0}, X_{2,
0})))\\
&-F_{1}(\widetilde{U}_{1}(s)+X_{1}(s, \omega, (X_{1, 0}, X_{2, 0})),
\widetilde{U}_{2}(s) +X_{2}(s, \omega, (X_{1, 0}, X_{2,
0})))]ds\|_{\mathcal{C}_{\beta, \alpha}^{1,
+}}\\
\leq& \sup\limits_{t\geq 0}e^{\beta
t}\|\int_{\infty}^{t}e^{-A_{1}(t-s)}[F_{1}(U_{1}(s) +X_{1}(s,
\omega, (X_{1, 0}, X_{2, 0})), U_{2}(s)+X_{2}(s, \omega, (X_{1, 0},
X_{2,
0})))\\
&-F_{1}(\widetilde{U}_{1}(s)+X_{1}(s, \omega, (X_{1, 0}, X_{2, 0})),
\widetilde{U}_{2}(s)+X_{2}(s, \omega, (X_{1, 0}, X_{2,
0})))]ds\|_{D(A_{1}^{\alpha})}\\
\leq& \sup\limits_{t\geq 0}e^{\beta
t}\int_{t}^{\infty}\|A_{1}^{\alpha}e^{-A_{1}(t-s)}[F_{1}(U_{1}(s)
+X_{1}(s, \omega, (X_{1, 0}, X_{2, 0})), U_{2}(s)+X_{2}(s, \omega,
(X_{1, 0}, X_{2,
0}))\\
&-F_{1}(\widetilde{U}_{1}(s)+X_{1}(s, \omega, (X_{1, 0}, X_{2, 0})),
\widetilde{U}_{2}(s)+X_{2}(s, \omega, (X_{1, 0}, X_{2,
0}))]\|ds \\
\leq &
\sup\limits_{t\geq0}L_{F}\lambda_{N}^{\alpha}\int_{t}^{\infty}e^{(-\lambda_{N}+\beta)(t-s)}(\|U_{1}
(s)-\widetilde{U}_{1}(s)\|_{\mathcal{C}_{\beta, \alpha}^{1,
+}}+\|U_{2}(s)-\widetilde{U}_{2}(s)\|_{\mathcal{C}_{\beta,
\alpha}^{2,
+}})ds\\
\leq&\sup\limits_{t\geq0}L_{F}\lambda_{N}^{\alpha}\int_{t}^{\infty}e^{(-\lambda_{N}+\beta)(t-s)}ds
\|\Psi-\widetilde{\Psi}\|_{\mathcal{C}_{\beta, \alpha}^{+}}\\
 \leq&
L_{F}\frac{\lambda_{N}^{\alpha}}{\beta-\lambda_{N}}\|\Psi-\widetilde{\Psi}\|_{\mathcal{C}_{\beta,
\alpha}^{+}}\,,
\end{array}
$$
and
\[
\begin{array}{lll}
&\|\mathfrak{F_{2}}(\Psi)-\mathfrak{F_{2}}(\widetilde{\Psi})\|_{\mathcal{C}_{\beta,
\alpha}^{2,
+}}\\
=&\|\int_{0}^{t}e^{-A_{2}(t-s)}[\Delta F_{2}(U_{1}(s), U_{2}(s),
\theta_{s}\omega)ds-\Delta F_{2}(\widetilde{U}_{1}(s),
\widetilde{U}_{2}(s), \theta_{s}\omega)ds]\|_{\mathcal{C}_{\beta, \alpha}^{2, +}}\\
=&\|\int_{0}^{t}e^{-A_{2}(t-s)}[F_{2}(U_{1}(s)+X_{1}(s, \omega, (X_{1, 0}, X_{2, 0})), U_{2}(s)+X_{2}(s, \omega, (X_{1, 0}, X_{2, 0})))\\
&-F_{2}(\widetilde{U}_{1}(s)+X_{1}(s, \omega, (X_{1, 0}, X_{2, 0})),
\widetilde{U}_{2}(s)+X_{2}(s,
\omega, (X_{1, 0}, X_{2, 0})))]ds\|_{\mathcal{C}_{\beta, \alpha}^{2, +}}\\
\leq& \sup\limits_{t\geq 0}e^{\beta
t}\|\int_{0}^{t}e^{-A_{2}(t-s)}[F_{2}(U_{1}(s)+X_{1}(s, \omega,
(X_{1, 0}, X_{2, 0})), U_{2}(s)+X_{2}(s, \omega, (X_{1, 0}, X_{2, 0})))\\
&-F_{2}(\widetilde{U}_{1}(s)+X_{1}(s, \omega, (X_{1, 0}, X_{2, 0})),
\widetilde{U}_{2}(s)+X_{2}(s,
\omega, (X_{1, 0}, X_{2, 0})))]\,ds\|_{D(A_{2}^{\alpha})}\\
\leq & \sup\limits_{t\geq 0}e^{\beta t}
\int_{0}^{t}(\frac{\alpha^{\alpha}}{(t-s)^{\alpha}}+
\lambda_{N+1}^{\alpha})e^{-\lambda_{N+1}(t-s)}
\|F_{2}(U_{1}(s)+X_{1}(s, \omega, (X_{1, 0}, X_{2, 0})), U_{2}(s)+\\
&X_{2}(s, \omega, (X_{1, 0}, X_{2,
0})))-F_{2}(\widetilde{U}_{1}(s)+X_{1}(s, \omega, (X_{1, 0}, X_{2,
0})), \\
&\widetilde{U}_{2}(s)+X_{2}(s,
\omega, (X_{1, 0}, X_{2, 0})))]\|\,ds\\
\leq & \sup\limits_{t\geq
0}\int_{0}^{t}\,(\frac{\alpha^{\alpha}}{(t-s)^{\alpha}}+\lambda_{N+1}^{\alpha})\,e^{(-\lambda_{N+1}
+\beta)(t-s)}(\|U_{1}
(s)-\widetilde{U}_{1}(s)\|_{\mathcal{C}_{\beta, \alpha}^{1, -}}+\|U_{2}(s)-\widetilde{U}_{2}(s)\|_{\mathcal{C}_{\beta, \alpha}^{2, +}})\,ds\\
\leq &\sup\limits_{t\geq
0}\int_{0}^{t}\,(\frac{\alpha^{\alpha}}{(t-s)^{\alpha}}+\lambda_{N+1}^{\alpha})\,e^{(-\lambda_{N+1}
+\beta)(t-s)}\,ds\,\|\Psi-\widetilde{\Psi}\|_{\mathcal{C}_{\beta,
\alpha}^{+}}\\
 \leq &
L_{F}(\frac{\lambda_{N+1}^{\alpha}}{\lambda_{N+1}-\beta}+\frac{C_{\alpha}}{(\lambda_{N+1}-\beta)^{1-\alpha}})
\|\Psi-\widetilde{\Psi}\|_{\mathcal{C}_{\beta, \alpha}^{+}}\,,
\end{array}
\]

where $C_{\alpha}=\alpha^{\alpha}\cdot\Gamma(1-\alpha)$\,, with
$\Gamma$ the Gamma function.
\par
By the definition of $\mathfrak{F}(\Psi)$ we have
\[
\begin{array}{ll}
&\|\mathfrak{F}(\Psi)-\mathfrak{F}(\widetilde{\Psi})\|_{\mathcal{C}_{\beta, \alpha}^{+}}\\
=&\|\mathfrak{F_{1}}(\Psi)-\mathfrak{F_{1}}(\widetilde{\Psi})\|_{\mathcal{C}_{\beta,
\alpha}^{1, +}}
+\|\mathfrak{F_{2}}(\Psi)-\mathfrak{F_{2}}(\widetilde{\Psi})\|_{\mathcal{C}_{\beta, \alpha}^{2, +}}\\
\leq&
L_{F}(\frac{\lambda_{N}^{\alpha}}{\beta-\lambda_{N}}+\frac{\lambda_{N+1}^{\alpha}}{\lambda_{N+1}-\beta}+\frac{C_{\alpha}}{(\lambda_{N+1}-\beta)^{1-\alpha}})
\|\Psi-\widetilde{\Psi}\|_{\mathcal{C}_{\beta, \alpha}^{+}}.
\end{array}
\]
Denote
\begin{equation}\label{e3.4}
k=L_{F}(\frac{\lambda_{N}^{\alpha}}{\beta-\lambda_{N}}+\frac{\lambda_{N+1}^{\alpha}}{\lambda_{N+1}-\beta}+\frac{C_{\alpha}}{(\lambda_{N+1}-\beta)^{1-\alpha}})<1
.
\end{equation}
Then the maping $\mathfrak{F}(\Psi)$ is contractive in
$\mathcal{C}_{\beta, \alpha}^{+}$ uniformly. From the uniform
contraction mapping principle, for each $U_{2}(0)\in
D(A_{2}^{\alpha})$, the mapping $\mathfrak{F}(\Psi)$ has a unique
fixed point, we still denoted it by $\Psi(\cdot)=\Psi(\cdot, \omega,
(X_{1, 0}, X_{2, 0}), U_{2}(0))\in \mathcal{C}_{\beta,
\alpha}^{+}$\,. So $\Psi(\cdot, \omega, (X_{1, 0}, X_{2, 0}),
U_{2}(0))\in \mathcal{C}_{\beta, \alpha}^{+}$ is a unique solution
of the system (\ref{e3.2}).
\end{proof}}
\end{lemma}

\begin{lemma}\label{l3.4}\quad{\it Take $\beta=\lambda_{N}+\frac{2L_{F}}{k}\lambda_{N}^{\alpha}\in(\lambda_{N}, \lambda_{N+1})$
as a positive real number. Let $$\Psi(\cdot)=\Psi(\cdot, \omega,
(X_{1, 0}, X_{2, 0});U_{2}(0))\in \mathcal{C}_{\beta, \alpha}^{+},$$
be the unique solution of the system (\ref{e3.3}). Then for every
$U_{2}(0)$, $\tilde{U}_{2}(0)\in D(A_{2}^{\alpha})$\,, we have the
following estimate
$$
\|\Psi(\cdot, \omega, (X_{1, 0}, X_{2, 0}), U_{2}(0))-\Psi(\cdot,
\omega, (X_{1, 0}, X_{2, 0}),
\tilde{U}_{2}(0))\|_{\mathcal{C}_{\beta, \alpha}^{+}} \leq
\frac{C}{1-k} \|U_{2}(0)-\tilde{U}_{2}(0)\|_{D(A_{2}^{\alpha})}.
$$
\begin{proof}
\quad By the same arguments as in Lemma (\ref{l3.3}), we have
$$
\|\mathfrak{F_{1}}(\Psi)\|_{\mathcal{C}_{\beta, \alpha}^{1, +}} \leq
L_{F}\frac{\lambda_{N}^{\alpha}}{\beta-\lambda_{N}}\|\Psi\|_{\mathcal{C}_{\beta,
\alpha}^{+}}\,,
$$
and
$$\|\mathfrak{F_{2}}(\Psi)\|_{\mathcal{C}_{\beta, \alpha}^{2,
+}}\leq
L_{F}(\frac{\lambda_{N+1}^{\alpha}}{\lambda_{N+1}-\beta}+\frac{C_{\alpha}}{(\lambda_{N+1}-\beta)^{1-\alpha}})
\|\Psi\|_{\mathcal{C}_{\beta,
\alpha}^{+}}+C\|U_{2}(0)\|_{D(A_{2}^{\alpha})}\,.$$ By the
definition of $\mathfrak{F}(\Psi)$ we have
$$
\|\mathfrak{F}(\Psi)\|_{\mathcal{C}_{\beta, \alpha}^{+}}\leq
\frac{C}{1-k} \|U_{2}(0)\|_{D(A_{2}^{\alpha})}\, ,$$ here $k$ is
defined as (\ref{e3.4}).\par Using the same argument as above we can
easily deduce our result as $$ \|\Psi(\cdot, \omega, (X_{1, 0},
X_{2, 0}), U_{2}(0))-\Psi(\cdot, \omega, (X_{1, 0}, X_{2, 0}),
\tilde{U}_{2}(0))\|_{\mathcal{C}_{\beta, \alpha}^{+}} \leq
\frac{C}{1-k} \|U_{2}(0)-\tilde{U}_{2}(0)\|_{D(A_{2}^{\alpha})}.
$$
This complete our proof
\end{proof}}
\end{lemma}
\par

For every $\zeta \in D(A_{2}^{\alpha})$\,, we define
\begin{equation}\label{e3.5}
\begin{array}{ll}
\mathfrak{f}(\zeta, (X_{1, 0}, X_{2, 0}), \omega):=& X_{1,
0}+\int_{\infty}^{0}e^{-A_{1}s}\Delta
F_{1}(U_{1}(s, \omega, (X_{1, 0}, X_{2, 0});(\zeta-X_{2, 0}))\,, \\\\
& U_{2}(s, \omega, (X_{1, 0}, X_{2, 0});(\zeta-X_{2, 0})),
\theta_{s}\omega)ds\,.
\end{array}
\end{equation}

\begin{theorem}(Invariant foliation)\label{t3.1}
\quad{\it Take
$\beta=\lambda_{N}+\frac{2L_{F}}{k}\lambda_{N}^{\alpha}\in(\lambda_{N},
\lambda_{N+1})$ as a positive real number. Then the foliation of the
system (\ref{e3.2}) exists. Moreover,
\begin{enumerate}
  \item A fiber is a graph of a $Lipschitz$ function, that is
  \begin{equation}\label{e3.6}
  \mathcal{W}_{\beta s}((X_{1, 0}, X_{2, 0}), \omega)=\{(\zeta, \mathfrak{f}(\zeta, (X_{1, 0}, X_{2, 0}), \omega))|\zeta\in D(A_{2}^{\alpha})\},
  \end{equation}
  where $(X_{1, 0}, X_{2,0})\in D(A_{1}^{\alpha})\times D(A_{2}^{\alpha})$, the function $\mathfrak{f}(\zeta, (X_{1, 0}, X_{2, 0}), \omega)$ defined as (\ref{e3.5}) is $Lipschitz$ continuous respect to $\zeta$ and the $Lipschitz$ constant $Lip\mathfrak{f}$ satisfies
$$Lip\mathfrak{f}\leq L_{F}\frac{\lambda_{N}^{\alpha}}{\beta-\lambda_{N}} \frac{C}{1-k},$$ where
$k$ is defined as (\ref{e3.4}).

  \item Exponentially approaching for positive time. That is , for any two points $(X_{1, 0}^{1}, X_{2, 0}^{1})$ and $(X_{1, 0}^{2}, X_{2, 0}^{2})$ in the same
  fiber $\mathcal{W}_{\beta s}((X_{1, 0}, X_{2, 0}), \omega)$\,. Then we have following estimate.

$$
  \begin{array}{ll}
  \|\Psi(t, \omega, (X_{1, 0}^{1}, X_{2, 0}^{1}))-\Psi(t, \omega, (X_{1, 0}^{2}, X_{2, 0}^{2}))\|_{D(A_{1}^{\alpha})
  \times D(A_{2}^{\alpha})}&\leq \frac{Ce^{-\beta t}}{1-k}
  \|X_{2, 0}^{1}-X_{2, 0}^{2}\|_{D(A_{2}^{\alpha})}\\\\
  &=O(e^{-\beta t})\,, \quad \forall\, t\rightarrow +\infty\,.
  \end{array}
 $$
 Here $C$ is a positive constant.

  \item The foliation is invariant, that is
  $$\Psi(t, \omega, \mathcal{W}_{\beta s}((X_{1, 0}, X_{2, 0}), \omega))\subset
  \mathcal{W}_{\beta s}(\Psi(t, \omega, (X_{1, 0}, X_{2, 0})), \theta_{t}\omega)\,.$$
\end{enumerate}}
\end{theorem}

\begin{proof}
\begin{enumerate}
  \item  Note that
$$\left( \begin{array}{c} \widetilde{X}_{1, 0}-X_{1, 0} \\\\ \widetilde{X}_{2, 0}-X_{2, 0} \end{array}
\right)=\left( \begin{array}{c} \int_{\infty}^{0}e^{-A_{1}s}\Delta
F_{1}(U_{1}(s), U_{2}(s), \theta_{s}\omega)ds
 \\\\ \widetilde{X}_{2, 0}-X_{2, 0}\end{array} \right) .$$
This implies that
 $$
 \begin{array}{ll}
 \widetilde{X}_{1, 0}&=X_{1, 0}+\int_{\infty}^{0}e^{-A_{1}s}\Delta F_{1}(U_{1}(s, \omega,
 (X_{1, 0}, X_{2, 0});U_{2}(0)),\\
 & U_{2}(s, \omega, (X_{1, 0}, X_{2, 0});U_{2}(0), \theta_{t}^{\epsilon}\omega)ds\\\\
 &=X_{1, 0}+\int_{\infty}^{0}e^{-A_{1}s}\Delta F_{1}(U_{1}(s, \omega, (X_{1, 0}, X_{2, 0});\widetilde{X}_{2, 0}-X_{2, 0}), \\\\
&U_{2}(s, \omega, (X_{1, 0}, X_{2, 0});\widetilde{X}_{2, 0}-X_{2, 0}), \theta_{t}^{\epsilon}\omega)ds,\\\\
   \end{array}
 $$
which is just $\mathfrak{f}(\zeta, (X_{1, 0}, X_{2, 0}), \omega)$ if
we take $\widetilde{X}_{2, 0}$ as $\zeta\in D(A_{2}^{\alpha})$\,.
Then from above discussion we have
$$\mathcal{W}_{\beta s}((X_{1, 0}, X_{2, 0}), \omega)=\{(\zeta, \mathfrak{f}(\zeta, (X_{1, 0}, X_{2, 0}), \omega))|\zeta\in D(A_{2}^{\alpha})\}. $$
Meanwhile, for every $\zeta,\, \widetilde{\zeta}$ in
$D(A_{2}^{\alpha})$\,, we have
$$
\begin{array}{ll}
&\|\mathfrak{f}(\zeta, (X_{1, 0}, X_{2, 0}),
\omega)-\mathfrak{f}(\widetilde{\zeta},
(X_{1, 0}, X_{2, 0}), \omega)\|_{C_{\beta, \alpha}^{1, +}}\\\\
=&\|U_{1}(t, \omega, (X_{1, 0}, X_{2, 0});\zeta-X_{2, 0})-U_{1}(t, \omega, (X_{1, 0}, X_{2, 0});\widetilde{\zeta}-X_{2, 0})\|_{C_{\beta, \alpha}^{1, +}}|_{t=0}\\\\
\leq&\|U_{1}(t, \omega, (X_{1, 0}, X_{2, 0});\zeta-X_{2, 0})-U_{1}(t, \omega, (X_{1, 0}, X_{2, 0});\widetilde{\zeta}-X_{2, 0})\|_{C_{\beta, \alpha}^{1, +}}\\\\
\leq&L_{F}\frac{\lambda_{N}^{\alpha}}{\beta-\lambda_{N}}\|\Psi(\cdot,
\omega, (X_{1, 0}, X_{2, 0});\zeta-X_{2, 0})-
\Psi(\cdot, \omega, (X_{1, 0}, X_{2, 0});\widetilde{\zeta}-X_{2, 0})\|_{C_{\beta, \alpha}^{+}}\\\\
\leq&L_{F}\frac{\lambda_{N}^{\alpha}}{\beta-\lambda_{N}}\frac{C}{1-k}\|\zeta-\widetilde{\zeta}\|_{D(A_{2}^{\alpha})}.
\end{array}
$$
This completes the proof of this part.

  \item  In this part we show that any two orbits in a given fiber are exponentially approaching each other in positive
  time.\
$$
 \begin{array}{ll}
&\|\Psi(\cdot)\|_{C_{\beta, \alpha}^{+}}
=\|U_{1}(\cdot)\|_{C_{\beta, \alpha}^{1,
+}}+\|U_{2}(\cdot)\|_{C_{\beta, \alpha}^{2,
+}}\\\\
\leq&\|e^{-A_{2}t}U_{2}(0)\|_{C_{\beta, \alpha}^{2,
+}}+\|\int_{0}^{t}e^{-A_{2}(t-s)}\Delta F_{2}(U_{1}(s), U_{2}(s),
\theta_{s}\omega)ds\|_{C_{\beta, \alpha}^{2, +}}\\\\
&+\|\int_{\infty}^{t}e^{-A_{1}(t-s)}\Delta F_{1}(U_{1}(s), U_{2}(s),
\theta_{s}\omega)ds\|_{C_{\beta}^{1,
+}}\\\\
\leq&C\|U_{2}(0)\|_{D(A_{2}^{\alpha})}+L_{F}\frac{\lambda_{N}^{\alpha}}{\beta-\lambda_{N}}\|\Psi(\cdot)\|_{C_{\beta,
\alpha}^{1, +}}
+L_{F}(\frac{\lambda_{N+1}^{\alpha}}{\lambda_{N+1}-\beta}+\frac{C_{\alpha}}{(\lambda_{N+1}-\beta)^{1-\alpha}})
\|\Psi(\cdot)\|_{C_{\beta, \alpha}^{2, +}}\\\\
\leq&C\|U_{2}(0)\|_{D(A_{2}^{\alpha})}+k\|\Psi(\cdot)\|_{C_{\beta,
\alpha}^{+}} ,
  \end{array}
  $$
  where $k$ is defined as (\ref{e3.4}) and $\Psi(\cdot)$ as (\ref{e3.3})\,. Then we have
  $$\|\Psi(\cdot)\|_{C_{\beta, \alpha}^{+}}\leq\frac{C}{1-k}\|U_{2}(0)\|_{D(A_{2}^{\alpha})}\,.$$
  By the definition of $\Psi(\cdot)$ and $U_{2}(0)$\,,
  we get
  $$\|\Phi(t\,, \omega\,, (\widetilde{X}_{1, 0}\,, \widetilde{X}_{2,
0}))-\Phi(t\,, \omega\,,
  (X_{1, 0}\,, X_{2, 0}))\|_{D(A_{1}^{\alpha})\times D(A_{2}^{\alpha})}\leq \frac{Ce^{-\beta t}}{1-k}\|U_{2}(0)\|_{D(A_{2}^{\alpha})}\,,  \forall\, t\geq 0\,.$$
  Take two point  $(X_{1, 0}^{1}\,, X_{2, 0}^{1})$ and $(X_{1, 0}^{2}\,, X_{2, 0}^{2})$ in the same
  fiber $\mathcal{W}_{\beta s}((X_{1, 0}\,, X_{2, 0})\,, \omega)$\,,
  along\\\\ with the Triangle Inequality we have
  $$\|\Phi(t\,, \omega\,, (X_{1, 0}^{1}\,, X_{2, 0}^{1}))-\Phi(t\,, \omega\,, (X_{1, 0}^{2}\,, X_{2, 0}^{2}))
  \|_{D(A_{1}^{\alpha})\times D(A_{2}^{\alpha})}\leq \frac{Ce^{-\beta t}}{1-k}\|U_{2}(0)\|_{D(A_{2}^{\alpha})}\,,
   \forall\, t\geq 0\,.$$ This completes the proof of this part.

  \item  In this part we prove the invariance of the foliation.
  Taking a fiber $\mathcal{W}_{\beta s}((X_{1, 0}\,, X_{2, 0})\,, \omega)$\,, we will show that $\Phi(\tau\,, \omega\,, \cdot)$ maps it into the fiber
  $\mathcal{W}_{\beta s}(\Phi(\tau\,, \omega\,, (X_{1, 0}\,, X_{2, 0})\,, \theta_{\tau}\omega))$\,.\\
  Let $(\widetilde{X}_{1, 0}\,, \widetilde{X}_{2, 0})\in \mathcal{W}_{\beta s}((X_{1, 0}\,, X_{2, 0}), \omega)$\,.
  Then $$\Phi(\cdot\,, \omega\,, (\widetilde{X}_{1, 0}\,, \widetilde{X}_{2, 0}))-\Phi(\cdot\,, \omega\,, (X_{1, 0}\,, X_{2, 0}))\in C_{\beta\,, \alpha}^{+}\,,$$
  which implies that
  $$\Phi(\cdot+\tau\,, \omega\,, (\widetilde{X}_{1, 0}\,, \widetilde{X}_{2, 0}))-\Phi(\cdot+\tau\,, \omega\,, (X_{1, 0}\,, X_{2, 0}))
  \in C_{\beta\,, \alpha}^{+}\,. $$
  By the cocycle property
  $$
   \begin{array}{ll}
  \Phi(\cdot+\tau\,, \omega\,, (\widetilde{X}_{1, 0}\,, \widetilde{X}_{2, 0}))=\Phi(\cdot\,, \theta_{\tau}\omega\,,
  \Phi(\tau, \omega\,, (\widetilde{X}_{1, 0}\,, \widetilde{X}_{2, 0})))\,, \\\\
  \Phi(\cdot+\tau\,, \omega\,, (X_{1, 0}\,, X_{2, 0}))=\Phi(\cdot\,, \theta_{\tau}\omega\,, \Phi(\tau\,, \omega\,, (X_{1, 0}\,, X_{2, 0})))\,,
    \end{array}
  $$
and the definition of foliation, we have
  $$
  \Phi(\tau\,, \omega\,, (X_{1, 0}\,, X_{2, 0}))\in \mathcal{W}_{\beta s}(\Phi(\tau\,, \omega\,, (X_{1, 0}\,, X_{2, 0}))\,, \theta_{\tau}\omega)\,.
  $$
This completes the proof of this part.
\end{enumerate}
\end{proof}

\begin{theorem}\label{t3.2} Under the same conditions of Theorem \ref{t3.1}, the foliation of the system (\ref{e1.1}) is
$$
\hat{\mathcal{{W}}}_{\beta s}((X_{1, 0}\,, X_{2, 0}),\,
\omega)=\mathcal{W}_{\beta s}((X_{1, 0}\,, X_{2, 0}),\,
\omega)+Z^{\epsilon}(t)\,.
$$
\end{theorem}
\begin{proof}
By the relationship between solutions of systems (\ref{e1.1}) and
(\ref{e3.2}), the system (\ref{e1.1}) has a Lipschitz continuous
invariant foliation under the same conditions of Theorem \ref{t3.1},
which is represented by
$$
\hat{\mathcal{{W}}}_{\beta s}((X_{1, 0}\,, X_{2, 0})\,,
\omega)=\mathcal{W}_{\beta s}((X_{1, 0}\,, X_{2, 0})\,,
\omega)+Z^{\epsilon}(t)\,.
$$
\end{proof}

\par
\par

\section{Approximation of invariant foliation}

\quad\, Now we consider an approximation of invariant foliation of
the system (\ref{e1.1}) when $\epsilon$ is sufficiently small
$$
\left\{\begin{array}{ll}
\frac{\partial u}{\partial t}-\sum_{k, j=1}^{n}\partial_{x_{k}}(a_{kj}(x)\partial_{x_{j}}u)+a_{0}(x)u=f(u)+\epsilon\dot{W_{1}} &\mbox{on D},\\\\
\frac{\partial u}{\partial t}+\sum_{k, j=1}^{n}\nu_{k}a_{kj}(x)\partial_{x_{j}}u+c(x)u=g(u)+\epsilon\dot{W_{2}} &\mbox{on $\partial D$},\\\\
u(0, x)=u_{0}(x) &\mbox{on $D\times \partial D$}.
\end{array}\right.
$$
The abstract form is
\begin{equation}\label{e4.1}
 d\hat{X}+A\hat{X}dt=F(\hat{X})dt+\epsilon dW, \quad \hat{X}_{0}=\hat{X}(0)\in D(A^{\alpha})\,,
 \end{equation}
 where $\hat{X}=(u, \gamma u)^{T}$\,, $F(\hat{X})=(f(u), g(\gamma u))^{T}$\,,  $W=(W_{1}, W_{2})$ and $X(0)=(u(0), \gamma u(0))^{T}$\,.
 \par
Consider a linear stochastic evolution equation
 \[dZ^{\epsilon}(t)+AZ^{\epsilon}(t)=\epsilon dW(t)\,.\]
Then $Z^{\epsilon}(\omega)=\epsilon Z(\omega)$\,, where $Z(\omega)$
is the stationary solution of \[dZ+AZ=dW(t)\,.\] Set
$X(t)=\hat{X}(t)-Z^{\epsilon}(t)$\,. Then $X(t)$ satisfies
\begin{equation}\label{e4.2}
dX+AXdt=F(X+Z^{\epsilon})dt\,.
\end{equation}
\par

\par

We propose an approach to approximate the random invariant foliation
by asymptotic analysis when $\epsilon$ sufficiently small. The
stable fiber of the invariant foliation for (\ref{e4.2}) passing
through $X_{0}=(X_{1, 0}, X_{2, 0})$ is denoted by
$$
\mathcal{W}_{\beta s}((X_{1, 0}, X_{2, 0}),
\omega)=\{\zeta+\mathfrak{f}^{\epsilon}(\zeta, (X_{1, 0}, X_{2, 0}),
\omega)|\zeta\in D(A_{2}^{\alpha}) \}\,,
$$
where
\begin{equation}\label{e4.3}
\begin{array}{lll}
&\mathfrak{f}^{\epsilon}(\zeta, (X_{1, 0}, X_{2, 0}),
\omega)\\
=&X_{1, 0}+\int_{\infty}^{0}e^{A_{1}s}[F_{1}(U_{1}(s, \omega, (X_{1,
0}, X_{2, 0})
;(\zeta-X_{2, 0}))+X_{1}(s, \omega, (X_{1, 0}, X_{2, 0})), \\
&U_{2}(s, \omega, (X_{1, 0}, X_{2, 0})
;(\zeta-X_{2, 0}))+X_{1}(s, \omega, (X_{1, 0}, X_{2, 0})), \theta_{s}\omega)\\
&-F_{1}(X_{1}(s, \omega, (X_{1, 0}, X_{2, 0})), X_{2}(s, \omega,
(X_{1, 0}, X_{2, 0})), \theta_{s}\omega)]\,.
\end{array}
\end{equation}

For the deterministic fiber $(i.e. \quad \epsilon=0)$ represented as
$$
\{\zeta+\mathfrak{f}^{d}(\zeta)|\zeta\in D(A_{2}^{\alpha})\}\,,
$$
where $\mathfrak{f}^{\epsilon}(\cdot, (X_{1, 0}, X_{2, 0}),
\omega):D(A_{2}^{\alpha})\rightarrow D(A_{1}^{\alpha})$ and
$\mathfrak{f}^{0}(\cdot, (X_{1, 0}, X_{2,
0})):D(A_{2}^{\alpha})\rightarrow D(A_{1}^{\alpha})$ are $Lipschitz$
mappings. We expand
\begin{equation}\label{e4.4}
\begin{array}{lll}
\mathfrak{f}^{\epsilon}(\zeta, (X_{1, 0}, X_{2, 0}), \omega)=\\\\
\mathfrak{f}^{d}(\zeta)+\epsilon \mathfrak{f}^{1}(\zeta, (X_{1, 0},
X_{2, 0}), \omega)+\epsilon^{2}\mathfrak{f}^{2}(\zeta, (X_{1, 0},
X_{2, 0}), \omega) +\cdots +\epsilon^{k}\mathfrak{f}^{k}(\zeta,
(X_{1, 0}, X_{2, 0}), \omega)+\cdots\,.
\end{array}
\end{equation}
As section $2$ we take $U=\widetilde{X}-X$ then for $U$ we have
follow equation
\begin{equation}\label{e4.5}
dU+AU=F(U+X+\epsilon Z)-F(X+\epsilon Z)\,.
\end{equation}
Write the solution of (\ref{e4.5}) in the form
\begin{equation}\label{e4.6}
U(t)=U^{(d)}(t)+\epsilon
U^{(1)}(t)+\cdots+\epsilon^{k}U^{(k)}(t)+\cdots,
\end{equation}
with the initial condition
$$
U(0)=\zeta+\mathfrak{f}^{\epsilon}(\zeta,
\omega)-X^{0}=\zeta-X^{0}+\mathfrak{f}^{0}(\zeta)+\epsilon^{1}\mathfrak{f}^{1}(\zeta,
\omega)+\cdots\,.
$$
We expand
\begin{equation}\label{e4.7}
X(t)=X^{(d)}(t)+\epsilon
X^{(1)}(t)+\cdots+\epsilon^{k}X^{(k)}+\cdots
\end{equation}
and
\begin{equation}\label{e4.8}
F(U(t)+X(t))=F(U^{(d)}(t)+X^{(d)}(t))+\epsilon
F^{\prime}_{|_{U^{(d)}(t)+X^{(d)}(t)}}\times(U^{(1)}(t)+X^{(1)}(t))+\cdots.
\end{equation}

Substituting  (\ref{e4.6}), (\ref{e4.7}), (\ref{e4.8}), into
equation (\ref{e4.5}), and equating the terms with the same power of
$\epsilon$, we get
$$
\left\{\begin{array}{ll}
\frac{dU^{(d)}(t)}{dt}=-AU^{(d)}(t)+F(U^{(d)}(t)+X^{(d)}(t))-F(X^{(d)}(t)),\\
U^{(d)}(0)=\zeta+\mathfrak{f}^{0}(\zeta)-X^{0},
\end{array}\right.
$$
and
$$
\left\{\begin{array}{ll}
\frac{dU^{(1)}(t)}{dt}&=-AU^{(1)}(t)+F^{\prime}_{|_{U^{(d)}(t)+X^{(d)}(t)}}\times(U^{(1)}(t)+X^{(1)}(t)
+Z(t))\\
&-F^{\prime}_{|_{X^{(d)}(t)}}\times(X^{(1)}(t)+Z(t))\,,\\
U^{(1)}(0)=\mathfrak{f}^{1}(\zeta, \omega)\,.
\end{array}\right.
$$
Then
$$
\left\{\begin{array}{ll}
\frac{dU^{(1)}(t)}{dt}=(-A+F^{\prime}_{|_{U^{(d)}(t)+X^{(d)}(t)}})U^{(1)}(t)+(F^{\prime}_{|_{U^{(d)}(t)+X^{(d)}(t)}}
-F^{\prime}_{|_{X^{(d)}(t)}})(X^{(1)}(t)+Z(t))\,,\\\\
U^{(1)}(0)=\mathfrak{f}^{1}(\zeta, \omega)\,,
\end{array}\right.
$$
here $X^{(d)}(t)$ and $X^{(1)}(t)$ satisfy the following equations
$$
\left\{\begin{array}{ll}
\frac{dX^{(d)}(t)}{dt}=-AX^{(d)}(t)+F(X^{(d)}(t))\,,\\\\
X^{(d)}(0)=X^{0}\,,
\end{array}\right.
$$
and
$$
\left\{\begin{array}{ll}
\frac{dX^{(1)}(t)}{dt}=-AX^{(1)}(t)+F^{\prime}_{|_{X^{(d)}(t)}}(X^{(1)}(t)+Z(t))\,,\\\\
X^{(1)}(0)=0\,.
\end{array}\right.
$$
Here we denote $F^{\prime}_{|_{U^{(d)}(t)+X^{(d)}(t)}}$ as the first
order Fr\'echet derivative of the function of $F$ evaluated at
$U^{(d)}(t)+X^{(d)}(t)$\,.
\par
With (\ref{e4.8}), the right hand side of (\ref{e4.3}) can be
written as
\begin{equation}\label{e4.9}
\mathfrak{f}^{\epsilon}(\zeta, \omega)=I_{0}+\epsilon I_{1}+R_{2}\,,
\end{equation}
where $R_{2}$ represents the remainder term and the other two terms
are as follows.
\begin{equation}\label{e4.10}
I_{0}=\pi_{1}X_{0}+\int_{\infty}^{0}e^{As}\pi_{1}(F(U^{(d)}(s,
\zeta-\pi_{2}X_{0}, X_{0})+X^{(d)}(s))-F(X^{(d)}(s)))ds\,,
\end{equation}
\begin{equation}\label{e4.11}
\begin{array}{lll}
I_{1}=&\int_{\infty}^{0}e^{As}\pi_{1}[F^{\prime}_{|_{U^{(d)}(t)+X^{(d)}(t)}}\times(U^{1}(s, \zeta-\pi_{2}X_{0}, X_{0})+X^{1}(s)+Z)\\\\
&-F^{\prime}_{|_{X^{(d)}(t)}}\times(X^{(1)}(s)+Z)]ds\,.
\end{array}
\end{equation}
Substituting (\ref{e4.4}) into (\ref{e4.9}), and matching the powers
in $\epsilon$\,, we get
$$
\mathfrak{f}^{d}(\zeta)=I_{0}=\pi_{1}X_{0}+\int_{\infty}^{0}e^{As}\pi_{1}(F(U^{(d)}(s,
\zeta-\pi_{2}X_{0}, X_{0})+X^{(d)}(s))-F(X^{(d)}(s)))ds
$$
and
$$
\begin{array}{lll}
\mathfrak{f}^{1}(\zeta, \omega,
X_{0})=I_{1}=&\int_{\infty}^{0}e^{As}\pi_{1}[F^{\prime}_{|_{U^{(d)}(t)+X^{(d)}(t)}}\times
(U^{1}(s, \zeta-\pi_{2}X_{0}, X_{0})+X^{(1)}(s)+Z)\\\\
&-F^{\prime}_{|_{X^{(d)}(t)}}\times(X^{(1)}(s)+Z)]ds\,.
\end{array}
$$
As a summary, we obtain the following result about approximating
invariant foliation for the random evolutionary equation with
dynamic boundary condition.

\begin{theorem}(Approximation of  invariant foliation)\label{t4.1} \quad{\it
Let $$\mathcal{W}_{\beta s}(X_{0}, \omega)=\{(\zeta,
\mathfrak{f}(\zeta, (X_{1, 0}, X_{2, 0}), \omega))|\zeta\in
D(A_{2}^{\alpha})\}$$ represent a stable fiber, passing through a
point $X_{0}$ of the invariant foliation for the random equation
$dX+AXdt=F(X+\epsilon Z)dt$\,. Assume that
\begin{enumerate}
  \item $F$ is twice continuously Fr\'echet differentiable.
  \item For $\beta=\lambda_{N}+\frac{2L_{F}}{k}\lambda_{N}^{\alpha}\in(\lambda_{N}, \lambda_{N+1})$
  and $C_{\alpha}=\alpha^{\alpha}\cdot\Gamma(1-\alpha)$, the following gap condition is satisfied
  $$L_{F}(\frac{\lambda_{N}^{\alpha}}{\beta-\lambda_{N}}+\frac{\lambda_{N+1}^{\alpha}}{\lambda_{N+1}-\beta}
  +\frac{C_{\alpha}}{(\lambda_{N+1}-\beta)^{1-\alpha}})<1\,. $$
  Then for $\epsilon$ sufficiently small, a fiber of the random invariant foliation can be approximated as
  $$
  \mathcal{W}_{\beta s}(X_{0}, \omega)=\{\zeta+\mathfrak{f}^{d}(\zeta)+\epsilon\mathfrak{f}^{1}(\zeta, \omega, X_{0})
  +R_{2}|\zeta\in D(A_{2}^{\alpha})\}\,.
  $$
  here $\|R_{2}\|\leq C(\omega)\epsilon^{2}$ with $C(\omega)<\infty$, a.s.\,,
  $$
  \mathfrak{f}^{d}(\zeta)=\pi_{1}X_{0}+\int_{\infty}^{0}e^{As}\pi_{1}(F(U^{(d)}(s, \zeta-\pi_{2}X_{0}, X_{0})
  +X^{(d)}(s))-F(X^{(d)}(s)))ds
  $$
  and
\[
\begin{array}{lll}
  \mathfrak{f}^{1}(\zeta, \omega, X_{0})=& \int_{\infty}^{0}e^{As}\pi_{1}[F^{\prime}_{|_{U^{(d)}(t)+X^{(d)}(t)}}\times
  (U^{(1)}(s, \zeta-\pi_{2}X_{0}, X_{0})+X^{(1)}(s)+Z) \\\\
  &-F^{\prime}_{|_{X^{(d)}(t)}}\times(X^{(1)}(s)+Z)]ds\,.
\end{array}
\]
\end{enumerate}}
\end{theorem}
\par
\section*{Acknowledgments} The author would like to Thank G. Chen for helpful discussions and suggestions.


\end{document}